%% file: cantor_space.tex
\documentclass[twoside,11pt,final]{entics} 
\usepackage{enticsmacro}
\usepackage{graphicx}
\usepackage[all]{xy}
\usepackage[T1]{fontenc}
\usepackage[UKenglish]{babel}

\usepackage{tikz,tikz-cd}

\usepackage{graphicx}
\usepackage{comment}
\usepackage{centernot}
\usepackage{mathtools,amssymb,amsmath}
\usepackage{hyperref}
\usepackage{tikzit}			
\input{markov.tikzstyles}

\usetikzlibrary{external,arrows,decorations.pathreplacing,decorations.pathmorphing,calligraphy,calc,positioning,fit,arrows.meta,decorations.markings,circuits.logic.US}	

\usepackage[capitalise]{cleveref}
\crefname{app}{Appendix}{Appendices}
\usepackage{ccicons}
\Crefname{thm}{}{} 
\newtheorem{notation}[thm]{Notation}

\newcommand{\newterm}[1]{\emph{\textbf{#1}}}

\renewcommand{\emptyset}{\varnothing} 

\newcommand{\N}{\mathbb{N}}

\newcommand{\op}{\mathrm{op}}
\newcommand{\cat}[1]{{\mathsf{#1}}}
\newcommand{\Setcat}{\mathsf{Set}}

\newcommand{\id}{\mathrm{id}} 		

\newcommand{\tensor}{\otimes}

\newcommand{\comp}{ 		
	\mathchoice{\,}{\,}{}{} 	
}

\DeclareMathOperator{\cop}{copy}
\DeclareMathOperator{\del}{del}

\newcommand{\cC}{\mathsf{C}}		
\newcommand{\cD}{\mathsf{D}}		
\newcommand{\inften}[1]{{#1}^{\otimes\infty}}
\newcommand{\inftenexp}[1]{{#1}_{\operatorname{expl}}^{\otimes\infty}}
\newcommand{\J}{J}
\newcommand{\pfin}[1][\J]{\mathsf{FinLO}(#1)}

\newcommand{\psh}[1]{[#1^{\op},\Setcat]}
\newcommand{\oppsh}[1]{[#1,\Setcat]^{\op}}
\DeclareMathOperator*{\colim}{\operatorname{colim}}

\newcommand{\free}[1]{\mathsf{Free}(#1)}
\newcommand{\inffree}[1]{\mathsf{Free}^{\infty}(#1)}

\newcommand{\oarrow}{\begin{tikzcd}[cramped,sep=small,ampersand replacement=\&]
{}\ar[r,oarrow]\&{}
\end{tikzcd}}
\tikzset{
  oarrow/.style={
    decoration={
      markings,
      mark=at position 0.5 with {%
        \node[draw,circle,inner sep=1pt,fill=white]{};%
      }
    },
    postaction={decorate},
    ->
  }
}

\newcommand{\finstoch}{\mathsf{FinStoch}}
\newcommand{\binstoch}{\mathsf{BinStoch}}
\newcommand{\gauss}{\mathsf{Gauss}}

\newcommand{\borelstoch}{\mathsf{BorelStoch}}
\newcommand{\canstoch}{\mathsf{CantorStoch}_{\mathsf{lc}}}
\newcommand{\stonestoch}{\mathsf{StoneStoch}_{\mathsf{lc}}}

\newcommand{\stoch}{\mathsf{Stoch}}

\newcommand{\nemp}{\operatorname{ne}}

\newcommand{\as}[1]{
	\def\relstate{#1}%
	\ifx\relstate\empty
		\text{a.s.}%
	\else
		{#1\text{-a.s.}}%
	\fi
}

\newcommand{\clopen}[1]{\mathsf{Clopen}(#1)}

\providecommand{\given}{\,|\,}			

\def\conf{MFPS 2026} 	
\volume{6}			
\def\volu{6}			
\def\papno{10}			
\def\mydoi{18679}
\def\lastname{Lorenzin, Zanasi} 
\def\runtitle{Approaching the Continuous from the Discrete} 
\def\copynames{A. Lorenzin, F. Zanasi}    
\begin{document}
\begin{frontmatter}
  \title{Approaching the Continuous from the Discrete: \\ an Infinite Tensor Product Construction\thanksref{ALL}} 						
 \thanks[ALL]{This work was partially supported by the ARIA Safeguarded AI TA1.1 programme. 
 We thank anonymous referees for helpful comments and suggestions.
 This work is licensed under the Creative Commons Attribution 4.0 International License (\ccLogo\hspace{1pt}\ccAttribution\ CC BY 4.0).
 }   
  \author{Antonio Lorenzin\thanksref{a}\thanksref{myemail}}	
   \author{Fabio Zanasi\thanksref{b}\thanksref{coemail}}		
   \address[a]{Independent researcher\\	
    Trento, Italy}  							
   \thanks[myemail]{Email: \href{mailto:a.lorenzin.95@gmail.com} {\texttt{\normalshape
        a.lorenzin.95@gmail.com}}} 
  \address[b]{Computer Science Department\\University College London\\
    London, UK} 
  \thanks[coemail]{Email:  \href{mailto:f.zanasi@ucl.ac.uk} {\texttt{\normalshape
        f.zanasi@ucl.ac.uk}}}
\begin{abstract}
	Increasingly in recent years, probabilistic computation has been investigated through the lenses of categorical algebra, especially via string diagrammatic calculi. Whereas categories of discrete and Gaussian probabilistic processes have been thoroughly studied, with various axiomatisation results, more expressive classes of continuous probability are less understood, because of the intrinsic difficulty of describing infinite behaviour by algebraic means.

	In this work, we establish a universal construction that adjoins infinite tensor products, allowing continuous probability to be investigated from discrete settings. Our main result applies this construction to $\finstoch$, the category of finite sets and stochastic matrices, obtaining a category of locally constant Markov kernels, where the objects are finite sets plus the Cantor space $2^\mathbb{N}$. Any probability measure on the reals can be reasoned about in this category. Furthermore, we show how to lift axiomatisation results through the infinite tensor product construction. This way we obtain an axiomatic presentation of continuous probability over countable powers of $2=\lbrace 0,1\rbrace$.
\end{abstract}
\end{frontmatter}
\section{Introduction}
Category-theoretic approaches to probabilistic computation have attracted growing interest in recent years, with applications ranging from evidential decision theory (\cite{dilavore2023evidential}) to random graphs (\cite{ackerman2024randomgraphs}) and active inference (\cite{tull2024activeinference}).
Their ability to highlight the underlying algebraic structures provides a rigorous semantics that enhances formal clarity, compositional methodologies, and admits an intuitive description in terms of \emph{string diagrams}~(\cite{piedeleuzanasi,selinger11graphical}). In particular, fundamental notions such as determinism, Bayesian inversion, and Bayesian updates can be studied algebraically in the diagrammatic language~(\cite{chojacobs2019strings,fritz2019synthetic,Jacobs_Kissinger_Zanasi_2021}).

Following this emerging line of research, recent developments have focused on providing \emph{complete axiomatisations} of categorical models of probability, via string diagrammatic theories. Completeness is key to guarantee that any semantic equality, and thus algebraic reasoning about the aforementioned notions, can in principle be derived in the diagrammatic calculus. In this context, axiomatising a (symmetric monoidal) category $\cat{C}$, which expresses our semantic domain, means to identify a set of generating string diagrams $\Sigma$ and equations $E$ such that the category $\free{\Sigma,E}$ freely obtained by the theory $(\Sigma,E)$ is isomorphic (or just equivalent) to $\cat{C}$. In other words, $\cat{C}$-morphisms may be regarded as $\Sigma$-diagrams quotiented by $E$. A first such result is the axiomatisation of the category of finite sets and stochastic matrices (with monoidal product given by disjoint sum) via the theory of convex algebras, presented in~\cite{fritz2009presentationcategorystochasticmatrices} following~\cite{Stone1949PostulatesFT}. Another such result is contained in \cite{stein2024graphical}, which studies a category of Gaussian probability using string diagrams expressing affine, relational, and quadratic behaviour. Most relevant to the present scope is the complete axiomatisation of a category of discrete probabilistic processes given in \cite{piedeleu2025completeaxiomatisation} (see also \cite[Section 4]{digiorgio2025parametric_iteration}). Specifically, the category under investigation, denoted by $\binstoch$, is the category of stochastic matrices between powers of $2=\lbrace 0,1\rbrace$, and the associated syntax is generated by probability distributions and Boolean operations.
Two related efforts generalise this description beyond $\binstoch$: \cite{sarkis2025gradedimprecise} extends the framework by introducing a graded structure to  encompass imprecise probabilities, while \cite{bonchi2025tapediagrams} notes that the use of an additional monoidal product enables an axiomatisation of $\finstoch$, the category of stochastic matrices between arbitrary finite sets.

To date, a fundamental gap in this research area is the absence of a string diagrammatic axiomatisation for categories of \emph{continuous} probability beyond the Gaussian case. A chief example is $\borelstoch$, whose objects are standard Borel spaces and whose morphisms are Markov kernels (generalisations of stochastic matrices to the continuous setting).
It is generally unclear how to encode infinite behaviour directly via a choice of generators and equations within the existing framework of string diagrammatic algebra. 

In this work we address such challenge through the use of \emph{infinite tensor products}, which are special limits of tensor (monoidal) products.\footnote{Throughout, we use ``tensor'' and ``monoidal'' synonymously, in order to both follow the terminology of \cite{fritzrischel2019zeroone} and adhere to the terminology of categorical probability.}
This approach is inspired by \cite{fritzrischel2019zeroone}, where it is shown that $\mathbb{R}$ is the infinite tensor product of finite sets in $\borelstoch$.
It also bears significant similarities to \cite{moss2026causalMarkov}, which explores the role of pro-completions in the context of Markov categories.\footnote{Note our work was developed independently from \cite{moss2026causalMarkov}. We give a detailed comparison of an approach based on pro-completions, as the one of \cite{moss2026causalMarkov}, and one based on finite approximation families, introduced in the present work, in Section~\ref{sec:finite_approximations}.}

To understand how infinite tensor products capture probabilistic meaning, we note that all probability measures on $\mathbb{R}$ can be seen as limits of probability measures on finite sets, by virtue of the celebrated Kolmogorov extension theorem. 
Our main contributions are the following:
\begin{enumerate}
	\item Introduce a universal construction that adjoins infinite tensor products (Theorem~\ref{thm:general_adjunction_ITP});
	\item Provide a concrete description of morphisms between infinite tensor products, in terms of \emph{finite approximation families} (\cref{sec:finite_approximations});
	\item Employ this framework to describe freely generated categories with an infinite tensor product, which admit a diagrammatic representation through plate notation (\cref{sec:stringdiagrams});
	\item Provide a characterisation of the category $\inften{\finstoch}$ obtained by adjoining infinite tensor products to $\finstoch$ (Theorem~\ref{thm:inften_of_finstoch}) in terms of Stone spaces and locally constant Markov kernels.
	\item Provide an axiomatic presentation of $\canstoch$, the restriction of $\inften{\finstoch}$ to powers of $2$ together with the Cantor space $2^{\mathbb{N}}$, its (countably) infinite power (Corollary~\ref{cor:binstoch_canstoch}).
\end{enumerate}
Regarding the fourth contribution, one might hope that the whole of $\borelstoch$ could be recovered as $\inften{\finstoch}$, i.e.\ by adjoining infinite tensor products to its discrete counterpart $\finstoch$.
However, this is overly ambitious: not because of probability measures, which can always be described as above, but because of the freeness of measurable functions, whose behaviour cannot be fully captured by finite operations. What we obtain instead is a subcategory of $\borelstoch$, which is nonetheless quite expressive: probability measures on $\mathbb{R}$ are in bijective correspondence with probability measures on $2^{\mathbb{N}}$, thus showing that $\canstoch$ is rich enough to encompass all probability measures on $\mathbb{R}$.
Additionally, this category includes nontrivial kernels that capture biased behaviour with respect to their inputs (see Example~\ref{ex:lc_Mker} for a detailed discussion).



\paragraph*{Outline} 
First we recall semicartesian categories, which provide the necessary structure to define infinite tensor products (\cref{sec:semicartesian}).
In \cref{sec:universal_inften} we give a universal construction of the category obtained by adjoining infinite tensor products to a semicartesian category.
\cref{sec:finite_approximations} restricts the discussion to semicartesian categories with cancellative deletions, where morphisms between infinite tensor products are determined by finite approximations.
String diagrams and monoidal theories are discussed in \cref{sec:stringdiagrams}, while \cref{sec:markov} focuses on Markov categories.
\cref{sec:stone} presents the main result, giving an explicit description of the universal construction for $\finstoch$ and $\binstoch$.
The developed syntax is applied in \cref{sec:plate_chains} to a short use case on Markov chains.
Concluding remarks and future directions are briefly discussed in \cref{sec:conclusions}.
Proofs of auxiliary results are deferred to the appendices.

\section{Semicartesian Categories and Infinite Tensor Products}\label{sec:semicartesian}
For the sake of generality, we will first consider semicartesian categories (sometimes known as affine symmetric monoidal categories), as their theory is rich enough to allow the study of infinite tensor products (\cite{fritzrischel2019zeroone}).
From \cref{sec:markov} onward, we will restrict focus to Markov categories, broadly adopted for modelling probabilistic and statistical behaviour --- see, e.g., \cite{cornish2024stochastic,fritz2019synthetic,jacobs2021logical,jacobs2019causal_surgery}.
In fact, the examples introduced below are all Markov categories.

\begin{definition}
	A \newterm{semicartesian category} is a symmetric monoidal category $(\cC,\tensor,I)$ in which the monoidal unit $I$ is terminal.
	The unique morphism of type $X \to I$ is denoted by $\del_X$ (in string diagrams, \hspace{-.5ex}\input{del.tikz}\hspace{-.5ex}) and called \emph{delete morphism}.
\end{definition}

\begin{example}
	The main example for our purposes is the semicartesian category $\finstoch$~\cite[Example 2.5]{fritz2019synthetic}, whose objects are finite sets and whose morphisms $f\colon X \to Y$ are functions $Y \times X \to \mathbb{R}^{\ge 0}$ satisfying $\Sigma_{y \in Y} f(y,x) = 1$, also known as stochastic functions.
	To better highlight the connection with probability theory, one generally employs the ``conditional notation'' and write $f(y\given x)$ in place of $f(y,x)$.
	In particular, $1= \Sigma_{y \in Y} f(y,x) =\Sigma_{y \in Y} f(y\given x)$ means that $f(-\given x)$ is a probability distribution on $Y$ for each $x \in X$.
	In other words, this category is the restriction of the Kleisli category of the distribution monad to finite sets.

	Composition is defined by $g\comp f (z \given x) := \sum_{y \in Y} g(z\given y) f(y \given x)$ for any composable morphisms $f \colon X \to Y$ and $g \colon Y \to Z$,
				while the tensor product is given by $f \otimes h (y,w \given x,z) := f(y\given x) h(w \given z)$ for any $f\colon X \to Y$ and $h \colon Z \to W$.
	In particular, the delete maps are given by $\del({\given x})=1$.
	In this paper, we will also consider $\finstoch_{\nemp}$, the full subcategory given by \emph{nonempty} finite sets. The reason for such a restriction will be explained in \cref{sec:finite_approximations}.
\end{example}
\begin{example}
	Another important example is $\binstoch$, the full subcategory of $\finstoch$ (and $\finstoch_{\nemp}$) given by finite powers of $2=\lbrace 0,1\rbrace$.
	This has become especially relevant after the work \cite{piedeleu2025completeaxiomatisation} provided a complete axiomatisation for such a category.
\end{example}

\begin{example}\label{ex:borelstoch}
	In \cite{fritzrischel2019zeroone}, where infinite tensor products were originally defined, the main example is $\borelstoch$, which generalizes $\finstoch$ outside the finite context.
	$\borelstoch$ is the category whose objects are standard Borel spaces (measurably isomorphic to a finite set, $\mathbb{Z}$, or $\mathbb{R}$ by Kuratowski's theorem) and whose morphisms are Markov kernels $f\colon X \to Y$, given by honest functions $f\colon \Sigma_Y \times X \to [0,1]$, where $\Sigma_Y$ is the $\sigma$-algebra of $Y$, such that $f(U\given -)\colon X \to [0,1]$ is measurable and $f(- \given x)\colon \Sigma_Y\to [0,1]$ is a probability measure.

	The measurability condition allows us to define a composition by integration: $g\comp f (U \given x) := \int_{Y} g(U\given dy) f(dy \given x)$.
	The tensor product is determined by the rule $f\otimes h (U\times V\given x,z) \coloneqq f(U\given x)h(V \given z)$.
	As before, the delete maps are given by $\del_X(I\given x) = 1$ for all $x \in X$ (where $I$ is a singleton).

	We will also use $\borelstoch_{\nemp}$ to denote the full subcategory of $\borelstoch$ without the empty set.
\end{example}
\begin{example}
	A more general version of $\borelstoch$ is $\stoch$, where objects are all measurable spaces. See \cite{fritz2019synthetic} for details.
\end{example}

\label{sec:inften}
Infinite tensor products were introduced in \cite{fritzrischel2019zeroone} to formulate zero--one laws in categorical language.
They have since become central in category-theoretic approaches to continuous probability; for instance, \cite{fritz2025empirical} employs them in discussing laws of large numbers. We now recall the basic setup and introduce some new terminology that will be convenient for later use. Throughout, let $\pfin$ denote the category whose objects are finite linearly ordered subsets of a \emph{countable} set $\J$, with a morphism $F \to F'$ whenever $F$ is an ordered subset of $F'$ (i.e., the inclusion preserves the order).
For the sake of simplicity, morphisms in $\pfin$ will often be denoted with the subset notation: $F\subseteq F'$.

\begin{definition}
	Let $\cC$ be a semicartesian category. A functor $X \colon \pfin^{\op} \to \cC$ is an \newterm{abstract infinite tensor product}\footnote{To be precise, we should say abstract \emph{countably} infinite tensor product. To avoid overloading notation, we will simply omit the adjective, and throughout infinite will always mean \emph{countably} infinite. Nonetheless, the full generality is expected to work, and this restriction is important only to our main application.} if $X$ sends $F$ to $X_F \coloneqq \bigotimes_{j\in F} X_{j}$, where $X_{j}$ is the image of $\lbrace j \rbrace \subseteq J$ via $X$, and the inclusion $F \subseteq F'$ to the morphism
\[
\pi_{F',F}\colon X_{F'}\to X_F
\]
which deletes all $j \in F' \setminus F$, and keeps the rest unchanged. 
Up to permutation, $\pi_{F',F}= \id_{X_F} \otimes \del_{X_{F'\setminus F}}$.

In this case, one also says that $X$ is an abstract infinite tensor product of the family $(X_j)_{j\in J}$.
\end{definition}

Intuitively, an abstract infinite tensor product allows one to tensor together infinitely many objects by looking at finite tensor products.

\begin{notation}
Whenever $X$ is an abstract infinite tensor product, its associated countable set $\J$ may be denoted by $\J_X$ for clarity.
\end{notation}

\begin{definition}
	Let $\cC$ be a semicartesian category.
	A \newterm{(concrete) infinite tensor product} $X=\bigotimes_{j\in \J} X_j$ of a family of objects $(X_j)_{j \in \J}$ in $\cC$ is a limit of the associated abstract infinite tensor product, provided this limit is preserved by $-\tensor Y$ for all $Y \in \cC$.

	The morphisms $X\to X_F$ given by the universal property will be called \newterm{finite marginalisations} and denoted by $\pi_F$.
\end{definition}

Note that such products are not strictly unique; rather, they are unique up to isomorphism. This subtle distinction  mirrors the one between strict and strong monoidal functors, and it is the reason we will work with the latter throughout.

\begin{definition}
	A semicartesian category \newterm{has infinite tensor products} if every abstract infinite tensor is realised, i.e.\ it is associated to a concrete infinite tensor product.
\end{definition}
In \cite{fritzrischel2019zeroone}, it was shown that $\borelstoch$ has infinite tensor products, and moreover $\mathbb{R}$ is isomorphic to any nontrivial infinite tensor product (this follows by the mentioned Kuratowski's theorem, since $\mathbb{R}$ is measurably isomorphic to any uncountable standard Borel space).


Throughout, symmetric monoidal functors are always meant to be \newterm{strong}, meaning that  their coherence morphisms are always isomorphisms.

\begin{definition}
A symmetric monoidal functor $\phi\colon \cC \to \cD$ between two semicartesian categories with infinite tensor products is \newterm{ITP-preserving} if $\phi(\bigotimes_{j\in \J} X_j)$ is an infinite tensor product of the family $(\phi(X_j))_{j\in \J}$.\footnote{In particular, the finite marginalisations $\pi_F$ are preserved by $\phi$.
As $\pi_F$ is a notation for any infinite tensor product, we will write $\phi(\pi_F)=\pi_F$ for simplicity.
}
\end{definition}
Before proceeding, we also consider the following technical requirement, necessary to establish the forthcoming results. In \cite{houghtonlarsen2021dilations}, semicartesian categories are called theories and the property below is called normality.
Following a now-established tradition in categorical modelling of probability \cite{chojacobs2019strings,fritz2019synthetic,piedeleuzanasi}, we will often depict morphisms of these categories using string diagrams.
\begin{definition}\label{def:cancellativedeletions}
	A semicartesian category $\cC$ has \newterm{cancellative deletions} if $f \otimes \del_X = g\otimes \del_X$ implies $f =g$ for all parallel morphisms $f,g$ and all objects $X$.
	In string diagrams,
	\[
	\input{fxdel.tikz}\,\,=\,\, \input{gxdel.tikz} \qquad \implies \qquad \input{f.tikz} \,\,=\,\,  \input{g.tikz}
	\]
	for all $f$ and $g$.
\end{definition}

In particular, surjectiveness of the marginalisations $\pi_{F',F}\colon X_{F'}\to X_F$ and $\pi_F\colon X\to X_F$ above is ensured.

The fact that $\finstoch_{\nemp}$, $\binstoch$, $\borelstoch_{\nemp}$ and $\stoch_{\nemp}$ satisfy this property follows by direct check:
whenever $f\otimes \del_X (U \given a,x) = g\otimes\del_X (U \given a,x)$, then $f(U\given a) = f(U\given a) \del_X(\given x)  = g(U\given a) \del_X(\given x) =g(U\given a)$ concludes the proof since $\del_X(\given x)=1$ for all $x$. Note that to ensure this property, the empty set is not an object of these categories.

\begin{remark}
Depending on taste, one may prefer to require cancellative deletions to hold whenever the domain of the deletion is not the initial object, in order to retain the empty set in the examples. In that case, one should additionally assume that tensoring with the initial object yields the initial object.\footnote{In this way, all infinite tensor products where the initial object occurs are initial as well.}
However, the treatment below would then have to single out the initial object in several places. To avoid this redundancy, we opted for a more direct approach. Nonetheless, the whole treatment could be adapted to this variant.	
\end{remark}



\section{A Universal Construction for Infinite Tensor Products}\label{sec:universal_inften}
We now focus on proving the following statement, which intuitively states that adjoining infinite tensor products satisfies a natural universal property.
This result is propaedeutic for our aims, but it is also of independent interest for other works using infinite tensor products \cite{chen2024aldoushoover,fritz2021definetti,fritzrischel2019zeroone}.

\begin{theorem}\label{thm:general_adjunction_ITP}
	Let $\cC$ be a semicartesian category.
	Then there exist a semicartesian category $\inften{\cC}$ with infinite tensor products and a fully faithful symmetric monoidal functor $\cC \to \inften{\cC}$, such that for every semicartesian category $\cD$ with infinite tensor products and every symmetric monoidal functor $\phi \colon \cC \to \cD$, there exists an ITP-preserving symmetric monoidal functor $\tilde{\phi}\colon \inften{\cC}\to \cD$ such that the following diagram commutes
	\[
	\begin{tikzcd}
	\cC \ar[rr, "\phi"]\ar[rd]&& \cD\\
	&\inften{\cC}\ar[ru,"\tilde{\phi}" below]&
	\end{tikzcd}
	\]
	Moreover, such a $\tilde{\phi}$ is unique up to natural isomorphism.
\end{theorem}
\begin{proof}
	This result follows from the well-known general fact that $\oppsh{\cC}$, the opposite category of functors $\cC \to \Setcat$, is the free completion of $\cC$.\footnote{One may prefer to consider the pro-completion instead, since infinite tensor products are cofiltered limits.}
	Moreover, such a category can be equipped with the (co)Day convolution: for $X$ and $Y$ in $\oppsh{\cC}$, we consider $X\otimes Y \coloneqq \int_{A,B} \cC(A\otimes B,-) \times X(A)\times Y(B)$, where $\int_{A,B}$ denotes the end with respect to $(A,B)$. In particular, the inclusion $\mathsf{cy}\colon \cC \to \oppsh{\cC}$ given by $A \mapsto \cC(A,-)$ is a fully faithful symmetric monoidal functor.

	We restrict to the full subcategory $\inften{\cC}\subset \oppsh{\cC}$ whose objects are infinite tensor products of objects of $\cC$.
	That these limits are indeed preserved under tensoring (i.e., by $-\otimes Y$) is clear because limits commute with limits, so in particular they commute with $\int_{A,B}$.
	We now consider the commutative diagram
	\[
	\begin{tikzcd}
	\cC \ar[rr, "\phi"]\ar[dr,"\iota"]\ar[dd, "\mathsf{cy}" left]&& \cD\ar[dd,"\mathsf{y}"]\\
	&\inften{\cC}\ar[ru,dashed,"\tilde{\phi}" below]\ar[dl,"\iota'"]&\\
	\oppsh{\cC}\ar[rr,"\hat{\phi}"]&& \psh{\cD} 
	\end{tikzcd}
	\]
	To prove the statement, it suffices to show that $\hat{\phi} \comp \iota'$ factors through $\mathsf{y}$, and moreover that this factorisation yields a symmetric monoidal functor. The first part follows directly from limit-preservation of $\tilde{\phi}$ and $\mathsf{y}$.

	Concerning the second property, we note that for two infinite tensor products $X$ and $Y$, we have $X \otimes Y \cong \lim_{F \in \pfin[\J_X]} \lim_{G \in \pfin[\J_Y]}\mathsf{cy}(X_F \otimes Y_G)$.
	We therefore conclude from the following natural isomorphisms:
	\begin{equation*}
	\begin{split}
	\hat{\phi} \left(\lim_{F} \lim_{G}\mathsf{cy}(X_F \otimes Y_G)\right)&\cong \lim_F \lim_G \hat{\phi}\comp \mathsf{cy} (X_F \otimes Y_G) \\
	&\cong \lim_F \lim_G \mathsf{y}\comp \phi (X_F \otimes Y_G)\\
	& \overset{(*)}{\cong}\lim_F \lim_G \mathsf{y}\comp \phi (X)_F\otimes \phi(Y)_G\\ 
	& \cong \mathsf{y}(\lim_F \lim_G \phi (X)_F\otimes \phi(Y)_G),
	\end{split}	
	\end{equation*}
	where $(*)$ is a consequence of strongness of $\phi$,\footnote{In particular, strongness ensures that the image of an infinite tensor product is indeed an infinite tensor product.} while the others follows from limit-preservation and the commutative diagram above. In particular, we stress that $\lim_F \lim_G \phi (X)_F\otimes \phi(Y)_G$ exists in $\cD$ by assumption, and since infinite tensor products are preserved under tensoring, it is isomorphic to $(\lim_F \phi(X)_F) \otimes( \lim_G \phi(Y)_G)$.

	Finally, $\inften{\cC}$ is a symmetric monoidal category with the wanted property. 
	It remains to show that it is indeed semicartesian. 
	As every object is a cofiltered (or inverse or projective) limit of representable functors, $\inften{\cC}(X, \cC(I,-)) \cong \inften{\cC} (\lim_F \cC(X_F,-), \cC(I,-)) \cong \colim_F \cC(X_F,I) \cong \colim_F \lbrace *\rbrace \cong \lbrace * \rbrace$, where the last isomorphism follows from the fact that the colimit is filtered.
\end{proof}
\begin{remark}
We remark that an analogous result could be stated for the pro-completion of $\cC$ (as studied in \cite{moss2026causalMarkov}) in place of $\inften{\cC}$, by requiring $\cD$ to have all cofiltered limits (and moreover that these are preserved under tensoring).
	The proof is analogous.
\end{remark}

\section{Finite Approximation Families}\label{sec:finite_approximations}
While theoretically pleasing, morphisms of $\inften{\cC}$ are difficult to describe succinctly. Indeed, by the Yoneda lemma, we have the following natural isomorphisms
\begin{equation}\label{eq:lim_colim_proobj}
	\cC^{\otimes \infty}(X,Y) \cong \lim_G \colim_F \cC(X_F,Y_G).
\end{equation}
(see, e.g., \cite{Grothendieck1960descent}).
However, unravelling this further proves challenging: 
An element of $\colim_F \cC(X_F,Y_G)$ is an equivalence class of morphisms such that, for every two representatives $f\colon X_F \to Y_G$ and $f' \colon X_{F'}\to Y_G$, there exists $H\supseteq F,F'$ such that 
\[
\begin{tikzcd}[row sep=tiny]
	&X_F\ar[dr,"f"]& \\
	X_H \ar[ur,"\pi_{H,F}"]\ar[dr,"\pi_{H,F'}" below left] && Y_G \\
	&X_{F'}\ar[ur,"f'" below]&
\end{tikzcd}
\]
commutes. 
In general, one has to content themselves with such a description. 
Indeed, even if $F=F'$, the two representatives $f$ and $f'$ can be distinct. For instance, in $\finstoch$, this can happen whenever $X_i=\emptyset$ for some $i \notin F$ (while $X_F\neq \emptyset$).

This motivates an alternative approach based on finite approximation families, which we now introduce. First, recall the notion of cancellative deletions from Definition~\ref{def:cancellativedeletions}, which ensures that all marginalisations are epimorphisms. 
Under this additional property, $f$ and $f'$ belong to the same equivalence class if and only if the commutative diagram above commutes for $H=F\cup F'$, disallowing distinct parallel representatives.
In more suggestive language, $X_F\to Y_G$ may now be viewed as a \emph{finite approximation} of $X\to Y$, encoding a finite amount of the data from which the latter is defined.
This yields the following description of morphisms.

\begin{definition}\label{def:finite_approximation}
	Let $\cC$ be a semicartesian category with cancellative deletions, and consider $X\colon \pfin[\J_X]^{\op}\to \cC$ and $Y\colon \pfin[\J_Y]^{\op} \to \cC$ two abstract infinite tensor products.
	A \newterm{finite approximation family} $f \colon X \to Y$ is given by an index set $\Lambda_f \subseteq \pfin[\J_X] \times \pfin[\J_Y]$ and a collection $(f_{F,G}\colon X_F\to Y_G)_{(F,G) \in \Lambda_f}$ subject to the following requests:
	\begin{itemize}
		\item \emph{(Naturality)} For every $F' \supseteq F$ and $G'\supseteq G$, if $(F,G),(F',G')\in \Lambda_f$, then the diagram
		\[
		\begin{tikzcd}
			X_{F'}\ar[r,"f_{F',G'}"]\ar[d,"\pi_{F',F}"'] & Y_{G'}\ar[d,"\pi_{G',G}"]\\
			X_F\ar[r,"f_{F,G}"] & Y_G
		\end{tikzcd}
		\]
		commutes;
		\item \emph{(Covering condition)} For every $G \in \pfin[\J_Y]$, there exists $F \in \pfin[\J_X]$ such that $(F,G)\in \Lambda_f$;
		\item \emph{(Hereditariness)} If $(F,G)\in \Lambda_f$, then $(F',G)\in \Lambda_f$ for every $F' \supseteq F$.
	\end{itemize}
	We moreover say that two finite approximation families $f,f'\colon X \to Y$ are \emph{equivalent} if $f_{F,G}=f'_{F,G}$ for all $(F,G)$ in the intersection of the index sets $\Lambda_f \cap \Lambda_{f'}$.
\end{definition}
Hereditariness above (together with cancellativity of deletions) ensures transitivity of the equivalence defined for finite approximation families; in particular, the covering condition also holds for $\Lambda_f \cap \Lambda_{f'}$.

We are finally ready to offer an explicit description of adjoining infinite tensor products.
To obtain a strict monoidal functor, we single out an element $*$, common to all countable sets $\J$ considered.

\begin{definition}\label{def:inften}
	Let $\cC$ be a semicartesian category with cancellative deletion, and let $*$ be a set.
	Then $\inftenexp{\cC}$ is the semicartesian category defined as follows:
	\begin{itemize}
		\item Each object is an abstract infinite tensor product $X \colon \pfin^{\op}\to \cC$, with $* \in J$;
		\item Morphisms $f\colon X \to Y$ are equivalence classes of finite approximation families;
		\item Composition is defined pointwise: $g\comp f$ is the finite approximation family given by
		\[ (g\comp f)_{F,H}\coloneqq g_{G,H}\comp f_{F,G}\]
		whenever there is a finite subset $G$ such that $(F,G)\in \Lambda_f$ and $(G,H)\in \Lambda_g$;
		\item The monoidal product of $X$ and $Y$ is given by
		\[
		\begin{array}{cccc}
		X \otimes Y\colon&\pfin[\J_X\cup \J_Y]^{\op}& \longrightarrow & \cC\\
		&F&\longmapsto &X_{F\cap \J_X}\otimes Y_{F\cap \J_Y},
		\end{array}
		\]
		and consequently, given $f \colon X \to Z$ and $g \colon Y \to W$, $f \otimes g$ is the morphism such that
		\[
		(f\otimes g)_{F,G} \coloneqq f_{F\cap \J_X,G \cap \J_Z} \otimes g_{F\cap \J_Y, G\cap \J_W}.
		\] (The monoidal unit is given by $\pfin[\lbrace * \rbrace]^{\op}\owns \lbrace * \rbrace \mapsto I\in \cC$).
	\end{itemize}
\end{definition}
The composition is thoroughly studied in Appendix~\ref{sec:inften_cat}, where it is shown that $\inftenexp{\cC}$ is a semicartesian category.
In particular, the request of having cancellative deletions ensures that the composition is well-defined. 
By construction, there exists a fully faithful symmetric monoidal functor $\cC \to \inftenexp{\cC}$, given by considering each object $X$ of $\cC$ as a functor $\pfin[\lbrace *\rbrace]\to \cC$ defined by $\lbrace *\rbrace \mapsto X$.
We note that $\inftenexp{\cC}$ is strict if $\cC$ is, since the monoidal product is described pointwise.\footnote{This follows from associativity of union. In particular, this motivates our choice of tensor product over other alternatives such as wedge sums (with respect to $\lbrace *\rbrace$), which are not strictly associative.} We are now ready to show that the approach based on pro-completions (Section~\ref{sec:universal_inften}) and the one based on finite approximation families are equivalent.
\begin{proposition}\label{prop:explicit_inften}
	$\inftenexp{\cC}$ is equivalent, via a strong monoidal functor, to $\inften{\cC}$.
\end{proposition}
\begin{proof}
	On objects, the functor is easily described (it sends an abstract infinite tensor product to its realisation).
	On morphisms, this boils down to understanding the set $\lim_G \colim_F \cC(X_F,Y_G)$.  
	Indeed, we note that $\colim_F \cC(X_F,Y_G)$ can be conveniently rewritten as the set
	\begin{multline*}
	\lbrace (F,(f_{F',G}\colon X_{F'}\to Y_G)_{F'\supseteq F}) \mid f_{F',G}\comp \pi_{F'',F'} = f_{F'',G} \text{ for all }F''\supseteq F' \supseteq F\rbrace\Big/\sim \\
\text{where } (F,(f_{F',G})_{F'\supseteq F}) \sim (H, (h_{H,G})) \iff f_{F \cup H,G} = h_{F\cup H,G}
	\end{multline*}
	thanks to cancellative deletions.
	This description in particular ensures hereditariness. 
	Naturality and covering conditions are also immediately recovered, and therefore we have an obvious bijection of morphisms. 
	Composition is given pointwise (i.e., at the level of $X_F\to Y_G$) for both descriptions of the colimit, hence functoriality follows.
	Finally, since in both categories infinite tensor products are limits of elements in $\cC$, which is a full monoidal subcategory of both, the functor is a strong monoidal one by universal property of limits.
\end{proof}

An important consequence of restricting to semicartesian categories with cancellative deletions is the following.
\begin{proposition}\label{prop:faithful}
	Let $\cC$ and $\cD$ be two semicartesian categories with cancellative deletions, and assume $\cD$ has infinite tensor products.
	If a symmetric monoidal functor $\phi\colon \cC\to \cD$ is faithful, then so is $\tilde{\phi} \colon \inften{\cC}\to \cD$.
\end{proposition}
\begin{proof}
	Let $f,g\colon X \to Y$ be two parallel morphisms in $\inften{\cC}$ such that $\tilde{\phi}(f)=\tilde{\phi}(g)$. Consider their marginalisations $\pi_G \tilde{\phi}(f) = \pi_G \tilde{\phi}(g) \colon \tilde{\phi}(X)\to \tilde{\phi}(Y)_G$.
	For any $G$, there exists $F\subseteq \J_X$ such that $\pi_G f = f_{F,G}\pi_F$ because $f$ is a finite approximation family, and analogously for $g$. Moreover, $F$ can be chosen for both $f$ and $g$ by covering condition and hereditariness, so also $\pi_G g =  g_{F,G} \pi_F$.
	Using this equality via $\tilde{\phi}$, we obtain
	\begin{multline*}
	\phi(f_{F,G}) \pi_F =\tilde{\phi}(f_{F,G}\pi_F ) = \tilde{\phi} (\pi_G f)= \pi_G \tilde{\phi}(f) =
	\pi_G \tilde{\phi}(g) = \tilde{\phi} (\pi_G g) = \tilde{\phi}(g_{F,G}\pi_F )=\phi(g_{F,G}) \pi_F
	\end{multline*}
	Since $\cD$ has cancellative deletions, $\pi_F$ is an epimorphism (note that $\pi_F = \id_{X_F} \otimes \del_{X_{\J_X \setminus F}}$ up to permutations), and therefore $\phi(f_{F,G}) =\phi(g_{F,G})$.
	By faithfulness of $\phi$, we conclude that $f_{F,G}=g_{F,G}$.
	Arbitrariness of $G$, together with naturality of finite approximation families, implies that $f$ and $g$ are equivalent finite approximation families, i.e.\ $f=g$.
\end{proof}

\section{String Diagrams and Monoidal Theories for Infinite Tensor Products}\label{sec:stringdiagrams}
In this section we study how to derive presentations by generators and equations for $\inftenexp{\cC}$ from those of~$\cC$. 
To this end, we introduce a string diagrammatic representation for morphisms of $\inftenexp{\cC}$,
reminiscent of the `plate' notation commonly used in the study of Bayesian networks, see e.g.\ \cite{barberBRML2012}.
In category-theoretic approaches to probability, this has recently been employed in \cite{chen2024aldoushoover}.
Our use of plate notation is similar in spirit, but it is tailored to ensure a formal representation of finite approximation families. 

We recall that a morphism $f\colon X\to Y$ in $\inftenexp{\cC}$ is given by a finite approximation family $(f_{F,G}\colon X(F)\to Y(G))_{(F,G)\in \Lambda_f}$, with $f_{F,G}$ in $\cC$, which we write as
\[
\input{f_FG_plate.tikz}
\]
in plate notation.
This morphism may also be denoted by \input{fXY_inften.tikz} for brevity, where the double wire of the input and output of $f$ indicate that $X$ and $Y$ are infinite tensor products.
For simplicity, we will generally omit $(F,G)\in \Lambda_f$, since the index set is determined by the finite approximation family.
This notation is compatible in the obvious way with sequential and parallel composition in $\inftenexp{\cC}$ (see \cref{fig:plate_compositions}). Moreover, it will be convenient for graphical reasoning to allow plates to disappear when infinite tensors are not involved, i.e. $(f_{F,G}\colon X_F\to Y_G)_{(F,G)\in \Lambda_f}$ is a \emph{bona fide} morphism of $\cC$. More precisely,
\begin{equation}\label{eq:disappearplate}
\input{f_FG_plate.tikz} \quad =\quad  \input{fXY.tikz}
\end{equation}
whenever both $X$ and $Y$ belong to the image of $\cC \to \inftenexp{\cC}$.
\begin{figure}[t]

\[
\input{sequential_plates.tikz} \quad = \quad \input{plate_sequential.tikz}
\]
\[
\input{parallel_plates.tikz}\quad = \quad \input{plate_parallel.tikz}
\]
\caption{Sequential and parallel composition of the plate notation.}\label{fig:plate_compositions}
\end{figure}

We can use plate notation to define a variation of the construction of a symmetric monoidal category from generators and equations, which encompasses infinite tensor products. This will be useful in the following sections to derive axiomatisation results.

First, recall that a \emph{symmetric monoidal theory} (SMT) is a pair $(\Sigma, E)$, where $\Sigma$ is a signature of generators $o \colon m \to n$ with an arity $m \in \mathbb{N}$ and coarity $n \in \mathbb{N}$, and $E$ is a set of equations between $\Sigma$-terms. A $\Sigma$-term $c$ of type $m \to n$ will be represented graphically as a box with $m$ dangling wires on the left and $n$ on the right, also written $\input{cdiagram.tikz}$. Formally, $\Sigma$-terms are freely obtained by  sequential and parallel compositions of the generators in $\Sigma$ together with the identity $\input{id.tikz} \colon 1 \to 1$, the symmetry $\input{swap.tikz} \colon 2 \to 2$, and the `empty' diagram $\input{empty-diag.tikz} \colon 0 \to 0$. Sequential composition of $\Sigma$-terms $\input{cdiagram.tikz}$ and $\input{ddiagram.tikz}$ is depicted as $\input{horizontal-comp.tikz}$, of type $m \to v$. Parallel composition of $\Sigma$-terms $\input{c1diagram.tikz}$ and $\input{c2diagram.tikz}$ is depicted as $\input{vertical-comp.tikz}$, of type $m_1 + m_2 \to n_1 + n_2$. Given an SMT $(\Sigma, E)$, the symmetric monoidal category $\free{\Sigma, E}$ freely generated by $(\Sigma, E)$ has objects the natural numbers and morphisms $m \to n$ the $\Sigma$-terms of type $m \to n$ quotiented by $E$ and by the laws of symmetric strict monoidal categories (see \cref{fig:lawssmc}), with sequential and parallel composition defined as on the corresponding $\Sigma$-terms. More details can be found e.g. in \cite{baez2018props,zanasi2018}.
\begin{figure}[t]
	\begin{equation*}
				\begin{array}{c}
					{\input{smc/sequential-associativity.tikz} = \input{smc/sequential-associativity-1.tikz}}
					\\[1em]
					\scalebox{1}{\input{smc/unit-right.tikz} = \input{cdiagramnotypes.tikz} = \input{smc/unit-left.tikz}}
					\\[1em]
					\scalebox{1}{\input{smc/parallel-associativity.tikz} = \input{smc/parallel-associativity-1.tikz}}
					\qquad
					\scalebox{1}{ \input{smc/parallel-unit-above.tikz} = \input{cdiagramnotypes.tikz} =  \input{smc/parallel-unit-below.tikz}}
					\\[3em]
					\scalebox{1}{\input{smc/interchange-law.tikz} = \input{smc/interchange-law-1.tikz} }
					\\[2em]
					\scalebox{1}{\input{smc/sym-natural.tikz}= \input{smc/sym-natural-1.tikz}}
					\qquad\quad
					\scalebox{1}{\input{smc/sym-iso.tikz} = \input{smc/id2.tikz}}
				\end{array}
			\end{equation*}
	\caption{Laws of symmetric strict monoidal categories}\label{fig:lawssmc}
\end{figure}
\begin{example}\label{ex:causcirc}
The SMC freely obtained by generators
\[
\input{del.tikz} \qquad \input{copy.tikz}\qquad \input{andgate.tikz}\qquad \input{notgate.tikz}\qquad \text{and}\qquad \input{pstates.tikz}\, \text{ for all }p\in [0,1],
\]
and set of equations as in \cite[Figure 4]{piedeleu2025completeaxiomatisation} is named $\cat{CausCirc}$ in \cite{piedeleu2025completeaxiomatisation}, as we may regard its string diagrams as causal circuits. Also in \cite{piedeleu2025completeaxiomatisation} it is proven that $\cat{CausCirc} \cong \binstoch$, meaning that such an SMT \emph{axiomatises} $\binstoch$.
An alternative axiomatisation of $\binstoch$ was provided in \cite[Section 4]{digiorgio2025parametric_iteration}.
\end{example}

We note that $\free{\Sigma,E}$ is semicartesian with cancellative deletions if and only if both of the following are satisfied:
\begin{enumerate}
	\item There is only one $\Sigma$-term (up to $=_E$) of type $n\to 0$, for any $n\in \mathbb{N}$;
	\item Given two $\Sigma$-terms $f$ and $g$, we have $f\otimes \del_n =_E g \otimes \del_n$ if and only if $f =_E g$.
\end{enumerate}
Under these assumptions, we may define how to freely obtain a category with infinite tensor products from $(\Sigma,E)$ as follows.

\begin{definition}\label{def:inffree} Given an SMT $(\Sigma, E)$ satisfying the two assumptions above, let $\inffree{\Sigma,E}$ be defined as the category (with infinite tensor products) whose objects are natural numbers together with an additional object $\infty$, and whose morphisms are given by finite approximation families of $\Sigma$-terms, subject to the equations of \cref{fig:plate_compositions} and \eqref{eq:disappearplate}.
\end{definition}
In particular, while this may be viewed as freely generating from generators with the addition of plates, in $\inffree{\Sigma,E}$ all morphisms are finite approximation families written using the plate notation. The plate notation can then be removed using \eqref{eq:disappearplate}.

The proof that $\inffree{\Sigma,E}$ is indeed a category follows by the same argument used to prove that $\inftenexp{\cC}$ is a category (see Appendix~\ref{sec:inften_cat}).

Regarding the tensor product, one sets $\infty\otimes N =\infty$ for any $N$ natural number or $\infty$, and the tensor product is obtained by reindexing the finite approximation families according to chosen fixed bijections $\mathbb{N}\sqcup \lbrace 0, 1, \dots , n-1\rbrace \cong \mathbb{N}$ and $\mathbb{N}\sqcup \mathbb{N}\cong \mathbb{N}$. These choices are sufficient to yield a coherent system: this follows from the fact that modifying a functor using a choice of isomorphisms for each object produces a naturally isomorphic functor.\footnote{For example, if we imagine $\inffree{\Sigma,E}$ inside $\inftenexp{\free{\Sigma,E}}$, we are redefining the functor $\otimes$ up to bijections of objects, e.g.~$1^{\mathbb{N}\sqcup \mathbb{N}}\cong 1^{\mathbb{N}}$.}
They are also necessary to ensure that associators, unitors and permutations are well-defined: for example, if we call $\phi$ the bijection $\mathbb{N}\sqcup \mathbb{N}\xrightarrow{\cong} \mathbb{N}$, then the associator is a reindexing with respect to the bijection
\[
\begin{tikzcd}[column sep=large]
	\mathbb{N}\ar[r,"\phi^{-1}"] & \mathbb{N}\sqcup \mathbb{N}\ar[r,"\id \sqcup \phi^{-1}"]& \mathbb{N}\sqcup \mathbb{N}\sqcup \mathbb{N}\ar[r,"\phi\sqcup \id"] & \mathbb{N}\sqcup \mathbb{N} \ar[r,"\phi"] & \mathbb{N}.
\end{tikzcd}
\]
Moreover, we can verify that $\inffree{\Sigma,E}$ indeed has infinite tensor products (combine Theorem~\ref{thm:general_adjunction_ITP} with Proposition~\ref{prop:explicit_inften}).

\begin{remark}\label{rem:reindexing}
It is worth noting that the parallel composition in \cref{fig:plate_compositions} is going to be influenced by \emph{reindexing} according to the chosen bijections for the tensor product of $\inffree{\Sigma,E}$.
To see exactly how, consider as an example two parallel morphisms $f,g\colon 0 \to \infty$, respectively corresponding to finite approximation families $f_{F}\colon 0 \to n_{f,F}$ and $g_{H}\colon 0\to n_{g,H}$, where $F,H\subseteq \mathbb{N}$ and the natural numbers $n_{f,F}$ and $n_{g,H}$ are associated to these choices.
Then the tensoring in $\inffree{\Sigma,E}$ of $f$ and $g$ is given by the finite approximation family
\[ (f\otimes g)_{\phi(F\sqcup H)} = f_{F}\otimes g_{H},\] 
where $\phi$ is the chosen bijection $\mathbb{N}\sqcup \mathbb{N}\xrightarrow{\cong} \mathbb{N}$.
\end{remark}

We now state that the infinite tensor product construction of Definition~\ref{def:inften}, when applied to a freely generated SMC $\free{\Sigma,E}$, is presented by the same $(\Sigma,E)$ via the $\inffree{-}$ construction above.

\begin{proposition}\label{cor:semicartesian_theory}
	Let $(\Sigma,E)$ be a symmetric monoidal theory such that $\free{\Sigma,E}$ is a semicartesian category with cancellative deletions.
	Then $\inffree{\Sigma,E}\simeq \inftenexp{(\free{\Sigma,E})}$ via an ITP-preserving symmetric monoidal functor.
\end{proposition}

\begin{proof}
By applying Theorem~\ref{thm:general_adjunction_ITP} (via Proposition~\ref{prop:explicit_inften}) and Proposition~\ref{prop:faithful}, we have a faithful ITP-preserving monoidal functor $\inftenexp{\free{\Sigma,E}} \to\inffree{\Sigma,E}$.
	Fullness is immediate by definition of $\inffree{\Sigma,E}$.
\end{proof}

We emphasise that $\inffree{\Sigma,E}$ and $\inftenexp{(\free{\Sigma,E})}$ are \emph{not} isomorphic.
In fact, these two categories reflect different trade-offs.
The former, $\inffree{\Sigma,E}$, admits a single infinite tensor product $\infty$, which may appear more natural at first glance, but then $\infty\otimes \infty = \infty$ requires to fix a chosen bijection $\mathbb{N}\sqcup \mathbb{N} \cong \mathbb{N}$, as mentioned above.
The latter, $\inftenexp{(\free{\Sigma,E})}$, instead allows multiple infinite tensor products, with the advantage that its tensor product does not rely on any choice of bijections.
Regardless of this distinction, in both categories one can reason equationally using the SMT $(\Sigma,E)$ together with the equations of \cref{fig:plate_compositions} and \eqref{eq:disappearplate}.

Combining Theorem~\ref{thm:general_adjunction_ITP} and Proposition~\ref{cor:semicartesian_theory}, one can prove the expected universal property for the free construction.

\begin{corollary}\label{cor:univproperty}
	Let $(\Sigma,E)$ be an SMT satisfying the two assumptions preceding Definition~\ref{def:inffree}, and consider a semicartesian category $\cD$ with infinite tensor products.
	Then a symmetric monoidal functor $\free{\Sigma,E}\to \cD$ extends to an ITP-preserving symmetric monoidal functor $\inffree{\Sigma,E}\to \cD$.
\end{corollary}
In other words, $\inffree{\Sigma,E}$ is initial among the semicartesian categories with infinite tensor products in which $(\Sigma,E)$ can be interpreted.
\begin{proof}
	By Theorem~\ref{thm:general_adjunction_ITP} and Proposition~\ref{prop:explicit_inften}, $\free{\Sigma,E}\to \cD$ extends to $\inftenexp{(\free{\Sigma,E})}\to \cD$. By composing this functor with the equivalence $\inffree{\Sigma,E}\simeq \inftenexp{(\free{\Sigma,E})}$ provided in Proposition~\ref{cor:semicartesian_theory}, the statement follows.
\end{proof}

\begin{remark}
The results of this section may be developed for $\inften{\cC}$ rather than $\inftenexp{\cC}$, being the two categories equivalent (Proposition~\ref{prop:explicit_inften}). 
More broadly, the present treatment can be adapted to the case where $\cC$ does not necessarily have cancellative deletions.
We privileged an exposition based on $\inftenexp{\cC}$ because finite approximation families encode useful information for providing a presentation by generators and equations, which are more implicit in the plain colimits formulation of $\inften{\cC}$ (\emph{cf.} \eqref{eq:lim_colim_proobj}).
\end{remark}

\section{Infinite Tensor Products and Markov Categories}\label{sec:markov}
As mentioned, our motivating examples are not just semicartesian but also Markov categories. Given the relevance of these structures in the categorical probability literature, it is of interest to show that whenever $\cC$ is a Markov category, then so is $\inftenexp{\cC}$.

\begin{definition}
	A \newterm{Markov category} is a semicartesian category $\cC$ where every object $X$ comes equipped with a cocommutative comonoid $\cop_X \colon X \to X\otimes X$ (whose counit is given by $\del_X$) compatible with the monoidal product. In string diagrams, $\cop_X$ is usually denoted by \input{copy.tikz} and the cocommutative comonoid equations read as follows.
	\[
	{\input{copy_commutative.tikz}}\qquad\qquad {\input{copy_associative.tikz}}
	\]
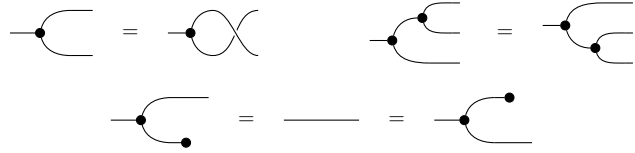
	\[
	{\input{del_unit.tikz}}
	\]
	Compatibility with the monoidal product is written as follows.
	\[
	\input{copyXY.tikz}\quad =\quad \input{copyXcopyY.tikz}
	\]
	(The delete maps are obviously compatible since they are the determined by terminality of $I$).
\end{definition}
All our examples ($\finstoch$, $\finstoch_{\nemp}$, $\binstoch$, $\borelstoch$, $\borelstoch_{\nemp}$, and $\stoch$) are Markov categories.
In $\finstoch$, $\finstoch_{\nemp}$, and $\binstoch$, we have $\cop_X (y,z\given x) \coloneqq 1$ if $x=y=z$ and $0$ otherwise. 
More details can be found in \cite{chojacobs2019strings,fritz2019synthetic}.

\begin{proposition}\label{prop:infinite_markov}
	Let $\cC$ be a Markov category with cancellative deletions. Then $\inftenexp{\cC}$ is a Markov category as well.
\end{proposition}
\begin{proof}
For every abstract infinite tensor product $X$, the collection of copy maps $(\cop_{X_F}\colon X_F\to X_F\otimes X_F)_{F\in \pfin[\J_X]}$ gives rise to a finite approximation family, thanks to compatibility with the monoidal product.
	We then define $\cop_X$ to be (the equivalence class of) this finite approximation family.
	Graphically,
	\[
	\input{copy_inften.tikz} \quad \coloneqq\quad \input{copy_plate.tikz}
	\]
	By \cref{fig:plate_compositions}, the cocommutative comonoid equations hold because they hold inside the plate. Similarly, compatibility with the monoidal product holds, therefore concluding the proof.
\end{proof}
In particular, by construction of the copy of $X$, a direct computation shows that $X$ is a Kolmogorov product in the sense of \cite{fritzrischel2019zeroone}.
\begin{remark}
	The reasoning of Proposition~\ref{prop:infinite_markov} can be adapted to the general case of $\cC$ being a Markov category (not necessarily with cancellative deletions), by using $\inften{\cC}$ instead of $\inftenexp{\cC}$ and by exploiting the description of morphisms offered in \eqref{eq:lim_colim_proobj} in place of finite approximation families.
\end{remark}

\section{A Stone Space Characterisation of $\inften{\finstoch}$}\label{sec:stone}
In this section we explicitly describe what $\inften{\finstoch}$ and $\inftenexp{\binstoch}$ are, therefore highlighting what information about the continuous setting can be recovered from $\finstoch$ and $\binstoch$. 
By Proposition~\ref{prop:faithful}, these infinite tensor products extensions can be viewed as subcategories of $\borelstoch$.
With this in mind, it suffices to understand what is the relevant notion of Markov kernels associated to $\inften{\finstoch}$ and $\inften{\binstoch}$.

Since $\inften{(-)}$ is simply a restriction of the pro-completion (see the proof of Theorem~\ref{thm:general_adjunction_ITP}), we can draw inspiration from the pro-category of the category of finite sets.
In fact, the discussion in this section can be generalised to a pro-completion procedure tailored to Markov categories; see \cite{moss2026causalMarkov} for further details.

With pro-completions in mind, we focus on \newterm{Stone spaces}, i.e.~compact Hausdorff spaces with a basis of clopen sets.
In particular, we will denote by {$\clopen{X}$} the clopen sets of a Stone space $X$.
These spaces are particularly well-known in the theory due to Stone duality, which states that $\clopen{-}$ yields a contravariant equivalence between the category of Stone spaces and continuous maps and the category of Boolean algebras and algebra homomorphisms.

In this setting, the most important Stone space for our purposes is the \newterm{Cantor space} $2^{\N}$, which, as the notation suggests, is given by the product of countably many copies of $2=\lbrace 0,1\rbrace$.
By Brouwer's theorem \cite{brouwer1910structure}, the Cantor space is the only non-empty Stone space that is second-countable and without isolated points. We now equip Stone spaces with a special type of Markov kernels.

\begin{definition}
	Let $X$ and $Y$ be two Stone spaces. A \newterm{locally constant Markov kernel} $X$ to $Y$, denoted by $f\colon X \oarrow Y$, is given by a function
	\[
	\begin{array}{ccccc}
		f &\colon& \clopen{Y}\times X & \to & [0,1]\\
		&& (U,x)& \mapsto & f(U\given x)
	\end{array}
	\]
	such that
	\begin{itemize}
		\item For all $U \in \clopen{Y}$, the function $f(U\given -)\colon X \to [0,1]$ is locally constant, i.e.\ for every $x \in X$ there exists an open $A\owns x$ such that $f(U\given x)= f(U\given y)$ for all $y\in A$;
		\item For all $x\in X$, the function $f(- \given x)\colon \clopen{Y}\to [0,1]$ is a finitely-additive probability measure, i.e.\ $f(Y\given x)=1$ and for every finite collection of disjoint clopen sets $\lbrace U_i\rbrace_{i=1}^n$, \[ f\left(\bigcup_{i=1}^n U_i\,\bigg|\, x\right)= \sum_{i=1}^n f(U_i\given x).\]
	\end{itemize}
\end{definition}
Stone spaces with locally constant Markov kernels form a Markov category denoted by $\stonestoch$ (see Appendix~\ref{sec:lcMarkov} for details). 
Readers interested in the monadic nature of these kernels may consult \cite[p.~5]{moss2026causalMarkov} ($\textsf{BKer}$ there is equivalent to $\stonestoch$ by Lemma 7, \emph{loc.~cit.}), where effect algebras are used to overcome a shortcoming of the Radon monad in the present context.

To the best of our knowledge, locally constant Markov kernels have not appeared previously in the literature. 
Their main relevance lies in providing a bridge between the discrete and continuous settings.
Nonetheless, they already capture meaningful nontrivial examples, as explored in Example~\ref{ex:lc_Mker}. 

\begin{remark}
	Locally constant Markov kernels can be defined as Markov kernels from $X$ to $Y$ (see Example~\cref{ex:borelstoch}), where both spaces are equipped  with the Baire $\sigma$-algebra (i.e., the $\sigma$-algebra generated by clopen sets), such that $f(U\given -)$ is locally constant for all clopen sets $U$. 
	This follows from Carath\'{e}odory's extension theorem, since a finitely-additive probability measure on clopen sets\footnote{In this setting, finite-additivity and $\sigma$-additivity are equivalent because Stone spaces are compact.} gives rise to a $\sigma$-additive probability measure on the Baire $\sigma$-algebra.
	In particular, this alternative description allows us to integrate against locally constant Markov kernels.
	Nonetheless, we opted for the presentation above so that in the treatment below every $U$ occurring in $f(U\given x)$ is a clopen instead of a general measurable subset.
\end{remark}

\begin{remark}
The emergence of local constancy is motivated by the strong connection between this property and finite approximation families.
Let us consider the Cantor space $2^{\mathbb{N}}$. 
For a locally constant Markov kernel, $f(U\given x)$ does not depend on the full information obtained by $x$, but only on an open set $A\owns x$ on which the kernel is constant.
Since every open set in $2^{\mathbb{N}}$ is a union of sets of the form $V=V_1 \times V_2 \times \dots \times V_n \times 2 \times \dots$, where only the first $n$ coordinates are proper subsets of $2$, it follows that $f(U\given x)$ depends only on finitely many coordinates of $x$. 
In other words,
\[
f(U\given x_1, x_2,\dots, x_n, \dots) = f(U \given x_1, x_2, \dots, x_n)
\]
for some $n$. This idea lies at the core of the proof of Theorem~\ref{thm:inften_of_finstoch}.
\end{remark}

\begin{example}\label{ex:nonlc_Mker}
	Not all Markov kernels are locally constant. For the sake of an example, we define a Markov kernel from the Cantor space $2^{\mathbb{N}}$ to $2$ given by
	\[
	f(U\given (x_n) ) \coloneqq \begin{cases}
		\sum_{n=0}^{\infty} \frac{2 x_n}{3^{n+1}} & \text{if }U=\lbrace 1 \rbrace \\
		1- \sum_{n=0}^{\infty} \frac{2 x_n}{3^{n+1}} & \text{if }U=\lbrace 0 \rbrace
	\end{cases}
	\]
	This is clearly not locally constant, since any choice of $(x_n)$ changes the value of $f(U\given (x_n))$, although $(x_n) \mapsto \sum_{n=0}^{\infty} \frac{2 x_n}{3^{n+1}}$ is the standard continuous embedding $2^{\mathbb{N}}\hookrightarrow [0,1]$.
\end{example}
Nonetheless, all continuous functions yield locally constant Markov kernels by means of delta measures (see Appendix~\ref{sec:lcMarkov} for details). We now present three examples of locally constant Markov kernels that demonstrate the expressive power of this concept.

\begin{example}\label{ex:lc_Mker}
	Note that $\clopen{2^{\mathbb{N}}}$ is the collection of sets of the form $U_F \times 2^{\mathbb{N}\setminus F}$, where $F\subseteq \mathbb{N}$ is finite and $U_F\subseteq 2^F$ is arbitrary.
	\begin{enumerate}
		\item\label{it:lc_p} A simple example of a locally constant Markov kernel is $p\colon I \oarrow 2^{\mathbb{N}}$, defined by
		\[
		p(U_F\times 2^{\mathbb{N}\setminus F}\given ) \coloneqq \frac{\# (U_F )}{2^{\#(F)}}
		\]
		where $\#(A)$ denotes the size of $A$.
		This construction is similar to the Lebesgue measure on the real line, in that it is translation-invariant: $p(U_F\times 2^{\mathbb{N}\setminus F}\given )= p(U_G \times 2^{\mathbb{N}\setminus G}\given)$ whenever $U_F$ and $U_G$ are in bijection.
		The fact that it is locally constant is immediate, since $I$ has only one element.
		\item A more involved example of a locally constant Markov kernel $f\colon 2^{\mathbb{N}}\oarrow 2^{\mathbb{N}}$ can be defined using the Markov kernel $p\colon I\oarrow 2^{\mathbb{N}}$ introduced in \cref{it:lc_p}. We set
		\[
		f(U\given x) \coloneqq \begin{cases}
			\frac{1}{2}p(U\given) + \frac{1}{2} & \text{if }x \in U\\
			\frac{1}{2}p(U\given) & \text{if }x \notin U
		\end{cases}
		\]
		Intuitively, the probability is biased toward the conditioning point $x$. The fact that $f(- \given x)$ is a finitely-additive probability measure follows from the same property of $p(- \given)$. Moreover, $f(U\given -)$ is constant on $U$ and on its complement, hence $f$ is locally constant.
		\item\label{it:midpoint} In a recent paper \cite{rischel2025universalmeasure}, Rischel describes a universal property of $\borelstoch$ formulated in terms of midpoint algebras $X\otimes X \to X$.\footnote{Midpoint algebras also appear implicitly in the axiomatisation of $\binstoch$ in \cite[Section 4]{digiorgio2025parametric_iteration}. In their notation, this corresponds to $\phi_A^{\frac{1}{2}}$, the compositon of the uniform probability on $\lbrace 0,1 \rbrace$ with the \emph{if-gate} $\phi_A\colon A \times \lbrace 0,1\rbrace \times A \oarrow A$.}
		Given our present interest in axiomatising continuous probability, it is worth observing that such algebras are in fact locally constant Markov kernels. 
		Indeed, they are given explicitly by the Markov kernels $f\colon X\times X\oarrow X$ defined by 
		\[
		f(U \given x,y) \coloneqq \frac{1}{2}\, \chi_U(x) +\frac{1}{2}\,\chi_U(y),
		\]
		where $\chi_U(z)=1$ if and only if $z\in U$. 
		By construction, $f(U\given -)$ is constant on $U \times U$, $U\times U^c$, $U^c\times U$, and $U^c\times U^c$; hence $f$ is locally constant, as claimed.
	\end{enumerate}
\end{example}

\begin{proposition}\label{prop:stonestoch_infinite}
	$\stonestoch$, the category of Stone spaces and locally constant Markov kernels, has all infinite tensor products. More explicitly, the infinite tensor product of $(X_j)_{j\in \J}$ is given by $\prod_{j\in \J}X_j$.
\end{proposition}
\begin{proof}
	First of all, we note that $X\coloneqq \prod_{j\in \J}X_j$ is a Stone space, with compactness ensured by Tychonoff's theorem.
	Now, let us consider $A$ another Stone space, and consider $f\colon A\oarrow X$ a locally constant Markov kernel.
	We claim that $f$ is determined by the marginals $f_F$, given by the composition
	\[
\begin{tikzcd}
	A \ar[r,oarrow,"f"] & X\ar[r,oarrow,"\pi_F"] & X_F = \prod_{j\in F} X_j,
\end{tikzcd}
	\]
	with $F$ any finite subset of $\J$.
	By compactness, for every clopen $U$ in $X$ there exist $F\in \pfin$ and a clopen $U_F\subseteq X_F$ such that $U= U_F\times \prod_{j\in \J \setminus F} X_j$.
	So $f(U \given a)= f_F (U_F \given a)$.
	In particular, $f$ is determined by its marginals, as claimed.

	Conversely, if we consider a family $(f_F \colon A \oarrow X_F)_{F\in \pfin}$ that is compatible in the sense that $\pi_{F,F'}\comp f_F = f_{F'}$, for every clopen set $U\subseteq \prod_{j\in J}X_j$, we can define $f(U\given a) \coloneqq f_F(U_F\given a)$, and compatibility ensures that this definition is independent of $F$.
	The limit property is therefore ensured.
	The preservation under tensoring follows by arbitrariness of the family of Stone spaces.
\end{proof}

Finally, we obtain our second main theorem, yielding a characterisation of infinite tensor products over $\finstoch$.

\begin{theorem}\label{thm:inften_of_finstoch}
The ITP-preserving symmetric monoidal functor
\[
\phi \colon \inften{\finstoch}\to \stonestoch,
\]
obtained by the inclusion $\finstoch \hookrightarrow \stonestoch$, is fully faithful.
Moreover, its essential image is given by infinite products of finite sets in the sense of topological spaces.
\end{theorem} 
In particular, by Brouwer's theorem \cite{brouwer1910structure}, any object in the essential image is isomorphic either to a finite set or to the Cantor space.
Moreover, $\phi$ preserves the copy map because in both settings the finite marginalizations are deterministic. In other words, this functor is strong gs-monoidal in the sense of \cite{fritz2023gsmonoidalfunctors}.
\begin{proof}
	First, recall that the functor is obtained by Theorem~\ref{thm:general_adjunction_ITP}. 
	Since the only object in $\finstoch$ that does not have cancellative deletions is the empty set, we obtain faithfulness from Proposition~\ref{prop:faithful} after restricting to $\finstoch_{\nemp}$ and then adjoining the empty set.

	We therefore only need to show fullness.
	Now, given a morphism between infinite products of finite sets $X=\prod_{j\in \J_X} X_j$ and $Y=\prod_{j \in \J_Y} Y_j$, we wish to show that $f\colon X\oarrow Y$ is uniquely determined by a finite approximation family $f_{F,G} \colon X_F \oarrow Y_G$.
	By Proposition~\ref{prop:stonestoch_infinite}, $f$ is uniquely determined by $(f_G)_{G\in \pfin[\J_Y]}$, with $f_G \coloneqq \pi_G f$.
	Let us now focus on $f_G$. Since it is locally constant, $f_G (U \given -)$ is constant on each member of a finite family $(A_i^U)_{i=1,\dots, n_U}$. Since $Y_G$ is finite, its subsets $U$ are finite, so we obtain a finite family of clopens $(A_i^U)_{i, \dots,n_U, U\subseteq Y_G}$.

	As any clopen in $X$ is of the form $A_F \times \prod_{j\in \J \setminus F} X_j$ for some $F$, there is a finite set $F$ common to all members of the finite family $(A_i^U)_{i,U}$, such that each $A_i^U$ can be written in such a cylindrical form.
	In particular, $f_G(U\given a) = f_{G} (U\given b)$ for any $b$ such that $\pi_F(a)=\pi_F(b)$.  In other words, $f_G$ factors through $\pi_F$ by defining, for any $\tilde{a}\in X_F$, $f_{F,G}(U\given \tilde{a})\coloneqq f_G(U\given a)$, where $a \in X$ is any element such that $\pi_F(a)=\tilde{a}$.
	It is now a direct check that the given definition of $f_{F,G}$ results in a finite approximation family, therefore ensuring fullness.
\end{proof}

\begin{corollary}\label{cor:binstoch_canstoch}
	$\inftenexp{\binstoch}$ is equivalent to $\canstoch$, the full subcategory of $\stonestoch$ in which objects are $2^{N}=\lbrace 0,1\rbrace^{N}$, where $N$ is any natural number or $\N$.
\end{corollary}

In particular, by combining Proposition~\ref{cor:semicartesian_theory}, Corollary~\ref{cor:binstoch_canstoch}, and Example~\ref{ex:causcirc}, we obtain a characterisation by generators and equations of $\canstoch$.

\begin{corollary}\label{cor:canstoch_axiomatisation}
$\canstoch$ is isomorphic to $\inffree{\Sigma,E}$, where $(\Sigma,E)$ is the symmetric monoidal theory of $\cat{CausCirc}$ (\emph{cf.} Example~\ref{ex:causcirc}).
\end{corollary}

Another important consequence of our results is that $\canstoch$ includes all probability measures on $\mathbb{R}$.
This follows from the fact that $\mathbb{R}$ is the infinite tensor power of $2$ in $\borelstoch$ \cite{fritzrischel2019zeroone}.
Indeed, since both $\mathbb{R}$ in $\borelstoch$ and $2^{\N}$ in $\canstoch$ are associated to finite approximation families of morphisms $I\to 2^n$ in $\finstoch$, we obtain $\borelstoch (I, \mathbb{R})\cong \canstoch (I,2^{\N})$. This observation suggests that $\canstoch$ is a rather expressive setting for studying continuous probability. We provide a simple use case below.

\section{A Use Case of Plate Notation: Markov Chains}\label{sec:plate_chains}

As an example of diagrammatic reasoning using plate notation, we briefly discuss Markov chains, which have been recently studied via Markov categories in \cite{fritz2025hidden}. We restrict here to the time-homogenous case, which defines a Markov chain with $n$ steps in $\binstoch$ as an endomorphism $f\colon X \to X$ inductively, as follows:
\[
\input{c1.tikz}\,\, \coloneqq \,\, \input{fX.tikz} \qquad\qquad \input{cn.tikz}\,\, \coloneqq \,\, \input{cn_def.tikz}
\]
For example, $c_3$ is given by the following string diagram:
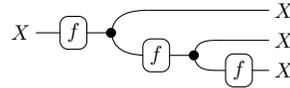
\[
\input{c3.tikz}
\]
The idea is that each copy of $X$ represents a random variable of the chain, and the use of the same $f$ for all steps translates to the request that $P(X_n \mid X_{n-1})= P(X_{n+1}\mid X_{n})$, i.e.\ that the obtained Markov chain is time-homogeneous.
Additional details can be found in \cite{fritz2025hidden}.

It can be verified that the $c_n$s define a finite approximation family, hence a morphism $c\colon X \to X^{\mathbb{N}}$, the infinite tensor power of $X$. As one may expect, $c$ is invariant under the addition of a precedent step in the inductive construction.
Note this property cannot be formulated in the discrete setting of $\binstoch$ (and of $\cat{CausCirc}$); however, we can easily prove it by diagrammatic reasoning in $\inffree{\Sigma,E}$, as justified by Corollary~\ref{cor:canstoch_axiomatisation}:
\begin{equation*}
	\begin{split}
	\input{cn_plate.tikz} \quad &=\quad \input{cn_f_plate.tikz} \quad =\quad \input{cn_plate_f.tikz}	\\[1.5ex]
	&= \quad \input{final_cn_plate_copy_f.tikz}
	\end{split}
\end{equation*}
where in the last step we also used reindexing, as prescribed by Remark~\ref{rem:reindexing}.

\section{Conclusion and Future Work}\label{sec:conclusions}
This paper focuses on introducing categorical methods for studying continuous probabilistic processes as limits of discrete ones. We do so by introducing a universal construction that associates to any semicartesian category $\cC$ a semicartesian category $\inften{\cC}$ with all (countably) infinite tensor products. When $\cC$ has cancellative deletions, $\inften{\cC}$ can be replaced (up to equivalence) by $\inftenexp{\cC}$, where in the latter morphisms are given by families of finite approximations.

Moreover, we discuss how axiomatic presentations for $\cC$ can be lifted to $\inftenexp{\cC}$, and introduce plate notation to manipulate as string diagrams $\inftenexp{\cC}$-morphisms. These results (in particular, Proposition~\ref{cor:semicartesian_theory}, Corollary~\ref{cor:univproperty}, and Corollary~\ref{cor:canstoch_axiomatisation}) are the first to provide tools to axiomatise Markov categories with infinite tensor products.

As main case study, we focus on $\finstoch$, whose infinite tensor construction $\inften{\finstoch}$ we characterise in terms of Stone spaces and locally constant Markov kernels.
Furthermore, the existence of an axiomatic presentation of the subcategory $\binstoch$ (having cancellative deletions) allows us to derive one for $\inftenexp{\binstoch}$. We can effectively use diagrammatic reasoning in $\inftenexp{\binstoch}$ to study any probability measure on $\mathbb{R}$. We give a very simple use case regarding Markov chains; clearly, this only scratches the surface of what can be studied in $\inftenexp{\binstoch}$, which remains to be explored in follow-up work. Nonetheless, it offers a necessary first step toward connecting semantic approaches with string diagrammatic methods for continuous probability.

Our approach suggests ways towards axiomatising Markov kernels beyond local constancy. 
A potentially fruitful approach is to closely examine the universal property of $\borelstoch$ established in \cite{rischel2025universalmeasure} (keeping in mind that midpoint algebras are not going to be sufficient, cf.~Example~\ref{ex:lc_Mker}\eqref{it:midpoint}).
Alternatively, one may try to describe general Markov kernels in terms of locally constant ones, possibly restricting to standard Borel spaces. Yet another direction may be to investigate a universal construction that guarantees disintegration of measures, in the sense of having conditionals in Markov categories (\cite{fritz2019synthetic}). 


Another interesting question is studying examples different from $\finstoch$ and $\binstoch$. Natural candidates are Gaussian probability ($\gauss$, as studied in~\cite{stein2024graphical}), Gaussian mixtures (\cite{torresruiz2025-mixtures}), or even non probabilistic examples of Markov categories, such as the category of finite sets and multivalued functions (\cite[Example 2.6]{fritz2019synthetic}).

Finally, we are interested in integrations of our approach (and of plate notation) with categorical approaches to probability that go beyond Markov categories, such as partial Markov categories~\cite{dilavore2024partial}, tape diagrams~\cite{bonchi2025tapediagrams}, and graded diagrams~\cite{sarkis2025gradedimprecise}.


\bibliographystyle{./entics}
\bibliography{references}

\appendix


\section{Finite Approximation Families Form a Semicartesian Category}\label{sec:inften_cat}
This appendix shows that the alternative universal construction $\inftenexp{\cC}$ offered in \cref{sec:finite_approximations} indeed yields a semicartesian category.
We start with a precise definition of the composition, briefly sketched in Definition~\ref{def:inften}.
\begin{definition}
	Let $\cC$ be a semicartesian category with cancellative deletion.
	Let $X$, $Y$ and $Z$ be abstract infinite tensor products and let $f\colon X \to Y$ and $g\colon Y\to Z$ be finite approximation families.

	We define the \emph{composition of finite approximation families} $gf\colon X \to Y$ by setting
	$\Lambda_{gf}\subseteq \pfin[\J_X]\times \pfin[\J_Z]$ to be the set of pairs $(F,H)$ for which there exists $G \in \pfin[\J_Y]$ such that $(F,G)\in \Lambda_f$ and $(G,H) \in \Lambda_g$, and
	\[
	gf_{F,H} \coloneqq g_{G,H} f_{F,G}
	\]
	for any meaningful choice of $G$.
\end{definition}
The proof that this is indeed a good notion of composition is easily shown.
Given two choices $G$ and $G'$, we can consider the following commutative diagram given by naturality
\[
\begin{tikzcd}[column sep=large]
	X_{\tilde{F}} \ar[r,"f_{\tilde{F},G\cup G'}"]\ar[d] & Y_{G \cup G'}\ar[r,"g_{G\cup G',H}"]\ar[d] & Z_{H}\ar[d,equal]\\
	X_{F}\ar[r,"f_{F,G}"] & Y_{G}\ar[r,"g_{G,H}"] & Z_{H}
\end{tikzcd}
\]
where $\tilde{F}$ is obtained by the covering condition. Cancellativity of deletions then implies surjectivity of $X_{\tilde{F}}\to X_{F}$, which in turn shows that the lower composition is completely determined by the upper composition, therefore concluding that $g_{G,H} f_{F,G} = g_{G',H} f_{F,G'}$.

\begin{lemma}\label{lem:composition_equivalence}
	Let $\cC$ be a semicartesian category with cancellative deletion.
	Then the equivalence of finite approximation families respects the composition.
\end{lemma}
\begin{proof}
	This follows from a reasoning similar to the one above: let $f,f'\colon X \to Y$ and $g,g'\colon Y \to Z$ be finite approximation families such that $f$ is equivalent to $f'$ and $g$ is equivalent to $g'$.
	Then, for each $H \in \pfin[\J_Z]$, there exists $G$ and $G'$ such that $g_{G,H}$ and $g'_{G',H}$ are defined.
	By equivalence of $g$ and $g'$ and hereditariness, we have $g_{G\cup G' , H} = g'_{G\cup G', H}$.
	Now, a similar reasoning for $f$ and $f'$ starting from $G \cup G'$ gives rise to a set $\tilde{F}$ such that $f_{\tilde{F},G\cup G'} = f'_{\tilde{F},G\cup G'}$.
	Combining the two together, we proved that for every $H$ there exists $\tilde{F}$ such that $gf_{\tilde{F},H} = g'f'_{\tilde{F},H}$.
	In particular, for each $(F,H)\in \Lambda_{gf}\cap \Lambda_{g'f'}$, we have $\tilde{F}\supseteq F$ such that $gf_{\tilde{F},H} = g'f'_{\tilde{F},H}$, and by cancellativity of deletion we conclude that $gf_{{F},H} = g'f'_{{F},H}$ as well.
\end{proof}

\begin{proposition}
	Let $\cC$ be a semicartesian category with cancellative deletions. Then $\inftenexp{\cC}$ is a semicartesian category.
\end{proposition}
\begin{proof}
	By Lemma~\ref{lem:composition_equivalence}, $\inftenexp{\cC}$ is a category.
	The remaining properties follow by direct check: its symmetric monoidal structure is inherited from that of $\cC$, with associators, unitors and permutations described pointwise, while semicartesianity follows from terminality of $I$ in $\cC$, by definition of finite approximation families.
\end{proof}

\section{Locally Constant Markov Kernels Form a Markov Category}\label{sec:lcMarkov}
This appendix complements \cref{sec:stone} and establishes that $\stonestoch$, the category of Stone spaces and locally constant Markov kernels, is a Markov category.
We first ensure that the composition is well-behaved.

\begin{lemma}\label{prop:composition}
	Let $f\colon X \oarrow Y$ and $g\colon Y\oarrow Z$ be locally constant Markov kernels. Then their composition
	\[
	gf(V\given x) \coloneqq \int_Y g(V\given y) f(dy\given x)
	\]
	is locally constant.
\end{lemma}
The proof is instructive as it sheds light on the ``finiteness'' of the notion.
\begin{proof}
	Since $g(V\given -)$ is locally constant on a Stone space, we can construct a finite clopen cover $\lbrace U_i\rbrace_{i=1}^n$ of $Y$ where all clopens are disjoint and $g(V\given -)_{| U_i}$ is constant.
	Then,
	\[
	\int_Y g(V\given y) f(dy\given x) = \sum_{i=1}^n g(V\given U_i) f(U_i \given x).
	\]
	We now define $A_i$ the clopen set such that $x \in A_i$ and $f(U_i\given -)$ is constant on $A_i$. Then, on $\bigcap_{i=1}^n A_i$, the whole summation is constant, therefore concluding the proof.
\end{proof}

To show that $\stonestoch$ is a Markov category, we need to ensure the existence of copy maps. In $\stoch$, these are determined by the diagonal maps $X \to X \times X$. With this in mind, we show the following.

\begin{lemma}\label{lem:continous_function}
	Any continuous function $X \to Y$ gives rise to a locally constant Markov kernel $X\oarrow Y$.
\end{lemma}
\begin{proof}
	That a continuous function yields a (continuous) Markov kernel is well-known, and follows by setting $\tilde{f}(U\given x) \coloneqq \delta_{f(x)}(U)$ for any continuous function $f\colon X \to Y$.
	Then for any clopen $U$, $f^{-1}(U)$ and its complement are sufficient to ensure that $\tilde{f}(U \given -)$ is locally constant.
	The fact that $\tilde{f}(- \given x)$ is a finitely-additive probability measure is immediate.
\end{proof}

\begin{proposition}
	The category $\stonestoch$, whose objects are Stone spaces and whose morphisms are locally constant Markov kernels, is a Markov subcategory of $\stoch$.
\end{proposition}

\begin{proof}
	By Lemma~\ref{prop:composition}, $\stonestoch$ is indeed a category.
	The tensor product of two locally constant Markov kernels is clearly a locally constant Markov kernel. Moreover, since all continuous maps give rise to locally constant Markov kernels (Lemma~\ref{lem:continous_function}), copy and delete maps also belong to $\stonestoch$, thus concluding the proof.
\end{proof}

\end{document}

%% file: markov.tikzstyles
\tikzstyle{morphism}=[fill=white, rounded corners=3pt, draw=black, shape=rectangle]
\tikzstyle{generic morphism}=[fill=white, draw=black, shape=rectangle, dashed]
\tikzstyle{small box}=[fill=white, draw=black,rounded corners=3pt, shape=rectangle, minimum width=0.5cm, minimum height=0.5cm]
\tikzstyle{medium box}=[fill=white, draw=black,rounded corners=3pt, shape=rectangle, minimum width=0.5cm, minimum height=0.8cm]
\tikzstyle{large morphism}=[fill=white, draw=black,rounded corners=3pt, shape=rectangle, minimum width=0.5cm, minimum height=1.2cm]
\tikzstyle{bn}=[fill=black, draw=black, shape=circle, inner sep=1.5pt]
\tikzstyle{bw}=[fill=white, draw=black, shape=circle, inner sep=1.5pt]
\tikzstyle{bin}=[fill=white, draw=black, shape=circle, inner sep=0pt]
\tikzstyle{effect2}=[fill=white, draw=black, regular polygon, regular polygon sides=3, minimum width=0.8cm, inner sep=0pt]
\tikzstyle{state}=[fill=white, draw=black, regular polygon, regular polygon sides=3, minimum width=0.8cm, shape border rotate=90, inner sep=0pt, rounded corners=3pt]
\tikzstyle{and}=[fill=white, draw=black, circuit logic US, and gate, minimum width=0.5cm, minimum height=0.8cm]
\tikzstyle{or}=[fill=white, draw=black, circuit logic US, or gate,minimum width=1cm, minimum height=1cm]
\tikzstyle{not}=[fill=white,circuit logic US, not gate, minimum width=1cm, minimum height=1cm]
\tikzstyle{xor}=[fill=white, draw=black, circuit logic US, xor gate, minimum width=0.5cm, minimum height=0.8cm]
\tikzstyle{if}=[trapezium, draw=black, fill=white, minimum width=6pt, minimum height=8pt, rotate=270]
\tikzstyle{coreg}=[draw, fill=white, rounded rectangle, rounded rectangle right arc=none, minimum height=1.2em, minimum width=1.4em, node font={\scriptsize}]
\tikzstyle{reg}=[draw, fill=white, rounded rectangle, rounded rectangle left arc=none, minimum height=.47cm, minimum width=.47cm, node font={\scriptsize}]
\tikzstyle{medium state}=[fill=white, draw=black, regular polygon, regular polygon sides=3, minimum width=1.3cm, inner sep=0pt, shape border rotate=90]
\tikzstyle{large state}=[fill=white, draw=black, regular polygon, regular polygon sides=3, minimum width=2.2cm, shape border rotate=180, inner sep=0pt]
\tikzstyle{wide state}=[fill=white, draw=black, shape=isosceles triangle, minimum width=0.8cm, shape border rotate=270, inner sep=1.4pt, minimum height=0.5cm, isosceles triangle apex angle=80]
\tikzstyle{wn}=[fill=white, draw=black, shape=circle, inner sep=1.5pt]
\tikzstyle{blue morphism}=[fill=white, draw={rgb,255: red,15; green,0; blue,150}, shape=rectangle, text={rgb,255: red,15; green,0; blue,150}, tikzit category=blue]
\tikzstyle{red morphism}=[fill=white, draw={rgb,255: red,150; green,0; blue,2}, shape=rectangle, text={rgb,255: red,150; green,0; blue,2}, tikzit category=red]
\tikzstyle{blue state}=[fill=white, draw={rgb,255: red,15; green,0; blue,150}, shape=circle, regular polygon, regular polygon sides=3, minimum width=0.8cm, shape border rotate=180, inner sep=0pt, text={rgb,255: red,15; green,0; blue,150}, tikzit category=blue]
\tikzstyle{blue node}=[fill={rgb,255: red,15; green,0; blue,150}, draw={rgb,255: red,15; green,0; blue,150}, shape=circle, tikzit category=blue, inner sep=1.5pt]
\tikzstyle{blue}=[text={rgb,255: red,15; green,0; blue,150}, tikzit draw={rgb,255: red,191; green,191; blue,191}, tikzit category=blue, tikzit fill=white, inner sep=0mm]
\tikzstyle{blue wide state}=[fill=white, draw={rgb,255: red,15; green,0; blue,150}, text={rgb,255: red,15; green,0; blue,150}, shape=isosceles triangle, minimum width=0.8cm, shape border rotate=270, inner sep=1.4pt, minimum height=0.5cm, isosceles triangle apex angle=80]
\tikzstyle{red node}=[fill={rgb,255: red,150; green,0; blue,2}, draw={rgb,255: red,150; green,0; blue,2}, shape=circle, inner sep=1.5pt]
\tikzstyle{Purple node}=[fill={rgb,255: red,120; green,0; blue,120}, draw={rgb,255: red,120; green,0; blue,120}, text={rgb,255: red,120; green,0; blue,120}, shape=circle, inner sep=1.5pt]
\tikzstyle{red}=[text={rgb,255: red,150; green,0; blue,2}, inner sep=0mm, tikzit fill=white, tikzit draw={rgb,255: red,191; green,191; blue,191}]
\tikzstyle{purple}=[text={rgb,255: red,150; green,0; blue,150}, inner sep=0mm, tikzit fill=white, tikzit draw={rgb,255: red,191; green,191; blue,191}]
\tikzstyle{white morphism}=[fill=white, draw=white, shape=rectangle, tikzit draw={rgb,255: red,139; green,139; blue,139}]
\tikzstyle{leak morphism}=[fill=white, draw={rgb,255: red,120; green,0; blue,85}, shape=rectangle, text={rgb,255: red,120; green,0; blue,85}, tikzit category=leak]
\tikzstyle{leak}=[text={rgb,255: red,120; green,0; blue,85}, inner sep=0mm, tikzit fill=white, tikzit draw={rgb,255: red,191; green,191; blue,191}, tikzit category=leak]
\tikzstyle{leak node}=[fill={rgb,255: red,120; green,0; blue,85}, draw={rgb,255: red,120; green,0; blue,85}, shape=circle, inner sep=1.5pt, tikzit category=leak]
\tikzstyle{horiz state}=[fill=white, draw=black, regular polygon, regular polygon sides=3, minimum width=1cm, shape border rotate=90, inner sep=0pt]
\tikzstyle{none_90}=[rotate=90]
\tikzstyle{none_-90}=[rotate=-90]

\tikzstyle{arrow}=[->]
\tikzstyle{dashed box}=[-, dashed]
\tikzstyle{blue arrow}=[-, draw={rgb,255: red,15; green,0; blue,150}, tikzit category=blue]
\tikzstyle{red arrow}=[-, draw={rgb,255: red,150; green,0; blue,2}, tikzit category=red]
\tikzstyle{purple arrow}=[->, draw={rgb,255: red,120; green,0; blue,120}, >=stealth, shorten <=2pt, shorten >=2pt]
\tikzstyle{protected purple arrow}=[->, draw={rgb,255: red,120; green,0; blue,120}, >=stealth, shorten <=2pt, shorten >=2pt, preaction={line width=1.8pt, white, draw}]
\tikzstyle{mapsto}=[{|->}]
\tikzstyle{double wire}=[-, draw, line width=0.8pt, white, preaction={-, draw, line width=1.8pt}]
\tikzstyle{curly brace}=[-, draw=none, tikzit draw={rgb,255: red,128; green,0; blue,128}]
\tikzstyle{protected}=[-, preaction={line width=1.8pt,white,draw}]
\tikzstyle{leak arrow}=[-, tikzit draw={rgb,255: red,150; green,0; blue,120}]
\tikzstyle{protected leak arrow}=[-, tikzit draw={rgb,255: red,150; green,0; blue,120}]
\tikzstyle{hollow arrow}=[-, very thin, white, preaction={line width=0.7pt,draw={rgb,255: red,120; green,0; blue,85}}, tikzit category=leak, tikzit draw={rgb,255: red,150; green,0; blue,120}]
\tikzstyle{protected hollow arrow}=[-, very thin, white, preaction={line width=0.7pt,draw={rgb,255: red,120; green,0; blue,85},preaction={line width=2.1pt,white,draw}}, tikzit category=leak, tikzit draw={rgb,255: red,150; green,0; blue,120}]
\tikzstyle{over arrow}=[-, black, preaction={draw=white, double}]
\tikzstyle{curly brace}=[-, decorate, decoration={brace,amplitude=5pt}]
\tikzstyle{inv curly brace}=[-, decorate, decoration={brace,amplitude=5pt,mirror}]

\tikzstyle{d-wire1 plate}=[-, double=red!20!white]
\tikzstyle{d-wire2 plate}=[-, double=red!32!white]
\tikzstyle{dotted_plate}=[-,rounded corners=3pt, densely dotted, draw=blue, fill opacity=0.4, fill=blue!50!white]
\tikzstyle{twire1}=[-, draw, line width=0.4pt, preaction={-, draw, line width=1.4pt, red!20!white, preaction={-, draw, line width=2.2pt}}]
\tikzstyle{twire2}=[-, draw, line width=0.4pt, preaction={-, draw, line width=1.4pt, red!32!white, preaction={-, draw, line width=2.2pt}}]

%% file: figures/del.tikz
\begin{tikzpicture}
	\begin{pgfonlayer}{nodelayer}
		\node [style=bn] (35) at (0.75, 0) {};
		\node [style=none] (36) at (-0.75, 0) {};
	\end{pgfonlayer}
	\begin{pgfonlayer}{edgelayer}
		\draw (35) to (36.center);
	\end{pgfonlayer}
\end{tikzpicture}

%% file: figures/f.tikz
\begin{tikzpicture}
	\begin{pgfonlayer}{nodelayer}
		\node [style=morphism] (15) at (0, 0) {$f$};
		\node [style=none] (16) at (-1, 0) {};
		\node [style=none] (17) at (1, 0) {};
	\end{pgfonlayer}
	\begin{pgfonlayer}{edgelayer}
		\draw (16.center) to (15);
		\draw (15) to (17.center);
	\end{pgfonlayer}
\end{tikzpicture}

%% file: figures/g.tikz
\begin{tikzpicture}
	\begin{pgfonlayer}{nodelayer}
		\node [style=morphism] (15) at (0, 0) {$g$};
		\node [style=none] (16) at (-1, 0) {};
		\node [style=none] (17) at (1, 0) {};
	\end{pgfonlayer}
	\begin{pgfonlayer}{edgelayer}
		\draw (16.center) to (15);
		\draw (15) to (17.center);
	\end{pgfonlayer}
\end{tikzpicture}

%% file: figures/smc/id.tikz
\begin{tikzpicture}
	\begin{pgfonlayer}{nodelayer}
		\node [style=none] (71) at (1.25, 0) {};
		\node [style=none] (72) at (-1.75, 0) {};
	\end{pgfonlayer}
	\begin{pgfonlayer}{edgelayer}
		\draw (72.center) to (71.center);
	\end{pgfonlayer}
\end{tikzpicture}

%% file: figures/notgate.tikz
\begin{tikzpicture}
	\begin{pgfonlayer}{nodelayer}
		\node [style=none] (16) at (-1, 0) {};
		\node [style=none] (17) at (1.25, 0) {};
		\node [style=not] (18) at (0, 0) {};
		\node [style=none] (21) at (0, 0) {};
	\end{pgfonlayer}
	\begin{pgfonlayer}{edgelayer}
		\draw [in=180, out=0, looseness=1.50] (16.center) to (21.center);
		\draw (18) to (17.center);
	\end{pgfonlayer}
\end{tikzpicture}

%% file: figures/pstates.tikz
\begin{tikzpicture}
	\begin{pgfonlayer}{nodelayer}
		\node [style=none] (17) at (1.25, 0) {};
		\node [style=coreg] (18) at (0, 0) {$p$};
	\end{pgfonlayer}
	\begin{pgfonlayer}{edgelayer}
		\draw (18) to (17.center);
	\end{pgfonlayer}
\end{tikzpicture}